\documentclass[pamm,a4paper,fleqn]{w-art}
\usepackage{times,cite,w-thm}
\usepackage[T1]{fontenc}
\usepackage[utf8]{inputenc}
\usepackage{graphicx}
\usepackage{multicol} 
\usepackage{amsmath,amssymb,amsthm}
\usepackage{mathtools}
\usepackage{enumitem}
\usepackage[british]{babel}

\usepackage{color}
\usepackage{xcolor}

\def\div{\operatorname{div}}

\def\RR{\mathbb{R}}

\def\eps{\varepsilon}
\def\LL{\mathbf{L}}

\def\bsig{\boldsymbol{\sigma}}
\def\Th{\mathcal{T}_h}
\def\Itau{\mathcal{I}_\tau}

\def\la{\langle}
\def\ra{\rangle}

\def\Th{\mathcal{T}_h}
\def\Vh{\mathcal{V}_h}

\def\Xh{\mathcal{X}_h}

\DeclarePairedDelimiter{\norm}{\|}{\|}
\DeclarePairedDelimiter{\snorm}{|}{|}

\def\u{\mathbf{u}}
\def\vv{\mathbf{v}}
\def\Du{\mathrm{D}\u}

\def\dt{\partial_t}

\begin{document}

\TitleLanguage[EN]
\title[Nonisothermal Cahn-Hilliard Navier-Stokes system]{Nonisothermal Cahn-Hilliard Navier-Stokes system}


\author{\firstname{Aaron}  \lastname{Brunk}\inst{1,}%
  \footnote{Corresponding author: e-mail \ElectronicMail{abrunk@uni-mainz.de}}}

\address[\inst{1}]{\CountryCode[DE]Johannes Gutenberg University Mainz}
\author{\firstname{Dennis} \lastname{Schumann}\inst{1}}

\AbstractLanguage[EN]
\begin{abstract}
In this research, we introduce and investigate an approximation method that preserves the structural integrity of the non-isothermal Cahn-Hilliard-Navier-Stokes system. Our approach extends a previously proposed technique \cite{brunk2024structurepreserving}, which utilizes conforming (inf-sup stable) finite elements in space, coupled with implicit time discretization employing convex-concave splitting.
Expanding upon this method, we incorporate the unstable $P_1|P_1$ pair for the Navier-Stokes contributions, integrating Brezzi-Pitkäranta stabilization. Additionally, we improve the enforcement of incompressibility conditions through grad div stabilization. While these techniques are well-established for Navier-Stokes equations, it becomes apparent that for non-isothermal models, they introduce additional coupling terms to the equation governing internal energy. 
To ensure the conservation of total energy and maintain entropy production, these stabilization terms are appropriately integrated into the internal energy equation.
\end{abstract}
\maketitle                   

\section{Introduction}

The non-isothermal Cahn-Hilliard-Navier-Stokes (CHNST) system has garnered increasing attention for investigating various phenomena, ranging from two-phase flows to fluid-phase-coupled interactions. These phenomena hold significant relevance in scientific and industrial domains, including additive manufacturing and inkjet printing \cite{van2017binary, Yang2020, dadvand2021advected}. For instance, in the modeling and simulation of powder bed fusion additive manufacturing (PBF-AM) processes, the non-isothermal CHNST model is utilized to portray coupled processes such as fluid-phase interaction, melt flow dynamics, and heat transfer \cite{Yang2020}.

The following system of partial differential equations describes the CHNST system under consideration:
\begin{align}
 \dt\phi &+ \u\cdot\nabla\phi - \div(\LL_{11}\nabla\mu + \LL_{12}\nabla\theta) = 0, \qquad\qquad\qquad\qquad\qquad \mu = -\gamma\Delta\phi + \partial_\phi \Psi(\phi,\theta), \label{eq:ac1}\\
 \dt e &+ \u\cdot\nabla e - \div(\LL_{12}\nabla\mu- \LL_{22}\nabla \theta ) - (\eta\Du- \bsig):\nabla\u = 0, \label{eq:ac2} \\
 \dt\u &+ (\u\cdot\nabla)\u - \div(\eta\Du - p\mathbf{I} - \bsig) = 0, \qquad\qquad\qquad\qquad\qquad \div(\u)=0. \label{eq:ac3}
\end{align}
The system described above is supplemented with periodic boundary conditions and initial conditions. In this context, $\phi$ represents a conserved phase-field, $\mathbf{u}$ signifies the flow velocity, $\theta$ denotes the inverse temperature, and $e\equiv e(\phi,\theta)$ represents the internal energy.
To complete the system, we incorporate the Helmholtz free energy density and the Korteweg stress, which are expressed as:
\begin{align*}
\tilde \Psi(\phi,\theta) =\Psi(\phi,\theta) + \tfrac{\gamma}{2}\snorm{\nabla\phi}^2, \qquad \bsig:=\tfrac{\gamma}{\theta}\nabla\phi\otimes\nabla\phi.
\end{align*}
From this, one can compute the internal energy and entropy according to \cite{Alt1990,Alt1992}, i.e. $e=\partial_\theta \tilde\Psi, \hat s(e(\theta,\phi),\phi) = s(\theta,\phi)=\theta e-\tilde\Psi$. 

The system described above is already formulated in terms of the inverse temperature $\theta$, which proves more advantageous for finite element discretization. The transformation from the original temperature to the inverse temperature is elucidated by Pawlow and Alt, as referenced in \cite{Alt1992}, and further discussed in \cite{brunk2023variational} for a more intricate non-isothermal Cahn-Hilliard-Allen-Cahn system. The general derivation of non-isothermal phase-field models from thermodynamically sound principles demands careful consideration, and for various modeling approaches, we refer to \cite{Charach1998, Fabrizio2006, Pawlow2016}.
For systems of this nature, approaches like those in \cite{Guo_2015, Sun2020} utilize finite differences in space alongside an Energy Quadratization (EQ) assumption, making the driving functional quadratic and thus conducive to standard time discretization techniques. However, this approach often comes at the cost of entropy relaxation, potentially leading to a departure from the original entropy in certain scenarios.

Stabilisation of unstable finite element pairs for the Navier-Stokes equation as well as other stabilisation methods as Grad-Div and SUPG are well-known in the incompressible Navier-Stokes context, cf. the monograph \cite{John2016}. In the case of lowest order $P_1|P_1$ element a quite easy stabilisation is the Brezzi-Pitkäranta stabilisation \cite{Brezzi1984}. 

In this study, we aim to extend a fully discrete method for the Cahn-Hilliard-Navier-Stokes system proposed in \cite{brunk2024structurepreserving} by incorporating unstable finite element pairs via the Brezzi-Pitkäranta stabilisation as well as the usual Grad-Div stabilisation. The focus of the extension is to preserve the underlying thermodynamic structure even at the discrete level. 

The organization of this work is as follows: Section 2 introduces pertinent notation and formulates a variational approach to System \eqref{eq:ac1}--\eqref{eq:ac3} suitable for finite element approximation. Subsequently, we present the fully discrete method and highlight the principal outcome, namely the preservation of total energy and entropy production. Section 3 demonstrates the theoretical framework through an appropriate convergence test. Finally, in Section 4, we conclude the study and offer insights into potential directions for future research.
 
\section{Notation and Main Result}
 
Before we present the new discretization method and main results, let us briefly introduce our notation and main assumptions.


%
\textbf{Notation:} The system \eqref{eq:ac1} -- \eqref{eq:ac3} is investigated on a finite time interval $(0,T)$ and bounded domain $\Omega$. 
For simplicity we consider a spatially periodic setting, i.e., $\Omega \subset \RR^d$, $d=2,3$ is a cube and identified with the $d$-dimensional torus $\mathcal{T}^d$. Moreover, functions on $\Omega$ are assumed to be periodic. 
We denote by $\langle \cdot, \cdot\rangle$ the scalar product on $L^2(\Omega)$, which is defined by
\begin{align*}
\la u, v \ra = \int_\Omega u \cdot v \quad \forall u,v \in L^2(\Omega) \quad\text{ with norm } \norm{u}_0^2:=\int_\Omega u^2.    
\end{align*}
We introduce the usual skew-symmetric formulation of $\mathbf{c}(\u,\vv,\mathbf{w}):=\la(\u\cdot\nabla)\vv,\mathbf{w} \ra$ via
\begin{equation*}
  \mathbf{c}_{skw}(\u,\vv,\mathbf{w}) = \tfrac{1}{2}\mathbf{c}(\u,\vv,\mathbf{w}) - \tfrac{1}{2}\mathbf{c}(\u,\mathbf{w},\vv)  
\end{equation*}
with the relevant property that $\mathbf{c}_{skw}(\u,\vv,\vv)=0$ even if $\u$ is not divergence-free.

We require the following assumptions for this work.
\begin{itemize}
    \item[] (A1) The interface parameters $\gamma$ is a positive constant.
    \item[] (A2) The viscosity function $\eta\equiv\eta(\phi,\theta)$ is strictly positive.
    \item[] (A3) The diffusion matrix $\LL\equiv \LL(\rho,\nabla\rho,\theta)\in\mathbb{R}^{2d \times 2d}$, i.e.
    $\begin{pmatrix}
        \LL_{11} & -\LL_{12} \\ - \LL_{12} & \LL_{22}
    \end{pmatrix},$ is symmetric \& strictly positive definite.
    \item[] (A4) For the driving potential $\Psi(\cdot,\cdot):\RR\times\RR_+\to \RR$ we assume that for fixed $\phi$ the potential $\Psi(\phi,\cdot):\RR_+\to\RR$ is concave and goes to infinity for $\theta\to 0$. For fixed $\theta\in\RR_+$ the potential $\Psi(\cdot,\theta)$ can be decomposed in a strictly convex and a strictly concave part, denoted by $\Psi_{vex}, \Psi_{cav}.$
\end{itemize}

\noindent\textbf{Variational Formulation:}
In \cite{brunk2024structurepreserving} it is shown that sufficiently smooth solutions obey the following variational formulation.

\begin{lemma}[\cite{brunk2024structurepreserving}]
   A sufficiently regular solution $(\phi,\mu,\theta,\u,p)$ of the system \eqref{eq:ac1}--\eqref{eq:ac3} fulfills the variational formulation 
		\hspace{-3em}\begin{align*}
		\la \dt\phi,\psi \ra &- \la \phi\u,\nabla\psi\ra + \la \LL_{11}\nabla\mu - \LL_{12}\nabla\theta,\nabla\psi \ra = 0, \\
		\la \mu,\xi \ra &- \gamma\la \nabla\phi,\nabla\xi \ra - \la \partial_\phi \Psi,\xi\ra = 0, \\
		\la \dt e,w \ra &+  \la \LL_{12}\nabla\mu - \LL_{22}\nabla\theta, \nabla w\ra  - \la \eta\snorm{\Du}^2, w\ra - \la \bsig\u,\nabla w \ra- \la \tfrac{\phi}{\theta}\nabla\mu- \boldsymbol{\sigma}\tfrac{\nabla\theta}{\theta},\u w\ra - \la s+\phi\mu,\u\cdot\nabla\tfrac{w}{\theta} \ra  = 0, \\ 
		\la \dt\u,\vv \ra &+ \mathbf{c}_{skw}(\u,\u,\vv) + \la \eta\Du,\mathrm{D}\vv \ra - \la \pi,\div(\vv) \ra +\la \tfrac{\phi}{\theta}\nabla\mu +  (s+\phi\mu)\nabla\tfrac{1}{\theta}-\boldsymbol{\sigma}\tfrac{\nabla\theta}{\theta},\vv \ra = 0,  \\
		0 &= \la \div(\u), q \ra
	\end{align*}
		for sufficiently regular test functions $\psi,\xi,\vv,w,q$ and  $\pi := p + e - \frac{s+\phi\mu}{\theta}$.
	\end{lemma}

This variational formulation allows to deduce the thermodynamics quantities by inserting simple test function.
\begin{theorem}[\cite{brunk2024structurepreserving}]
For a sufficiently regular solution $(\phi,\mu,\theta,\u,\pi)$ of \eqref{eq:ac1} -- \eqref{eq:ac3} conservation of mass and total energy as well as entropy production holds, i.e.
\begin{align*}
 \la \dt\phi,1 \ra &= 0, \qquad \la \dt (\tfrac{1}{2}\snorm{\u}^2+e(\phi,\theta)), 1\ra = 0, \\
 \la \dt s(\phi,\theta),1 \ra &= \norm{\sqrt{\eta\theta}\Du}_0^2 + \la (
    \nabla\mu, \nabla \theta)^\top,\LL (
    \nabla\mu, \nabla \theta)^\top  \ra = : \mathcal{D}_{\LL^*,\theta}(\u,\mu,\theta)\geq0. 
\end{align*}
\end{theorem}

\noindent\textbf{Time Discretization:}
We divide the time interval $[0,T]$ uniformly into intervals of size $\tau>0$, defining the time grid $\Itau:={t^0=0,t^1=\tau,\ldots, t^{n_T}=T}$, where $n_T=\tfrac{T}{\tau}$ represents the total number of time steps. The spaces $\Pi^1_c(\Itau;X)$ and $\Pi^0(\Itau;X)$ denote continuous piecewise-linear and piecewise-constant functions on $\Itau$ with values in the space or set $X$. Here, $g^{n+1}$, $g^n$, and $g^{n+1/2}$ refer to the new, old, and midpoint approximations of $g$, respectively, given by $(g^{n+1}+g^n)/2$. We introduce the time difference and discrete time derivative as follows
\begin{equation*}
d^{n+1}g = g^{n+1} - g^n, \quad d^{n+1}_\tau g = \tau^{-1}(g^{n+1}-g^n)=\tau^{-1}d^{n+1}g.
\end{equation*}

\textbf{Space Discretization:} For spatial discretization, we require that $\Th$ be a geometrically conforming partition of $\Omega$ into simplices that can be periodically extended to cover $\Omega$. We denote the space of continuous piecewise linear over $\Th$, as well as the space of mean-free and positive piecewise-linear functions over $\Th$, as follows
\begin{align*}
   \hspace{-3em} \Vh := \{v \in H^1(\Omega)\cap C^0(\bar\Omega) : v|_K \in P_1(K) \quad \forall K \in \Th\}, \qquad \Xh := \Vh^d, \qquad \Vh^+ := \{v \in \Vh : v(x) > 0,\; \forall x\in\Omega\}.
\end{align*}

We denote the convex-concave splitting by the following abbreviation
\begin{equation*}
 \Psi(\phi_h^{n+1},\phi_h^n,\theta_h^{n+1}) = \Psi_{vex}(\phi_h^{n+1},\theta_h^{n+1}) + \Psi_{cav}(\phi_h^{n},\theta_h^{n+1})   
\end{equation*}
and we will use the notation $e(\phi_h^{n+1},\theta_h^{n+1})=:e_h^{n+1}$ and similarly for $s,\Psi$.

 We then propose the fully discrete time-stepping method for the CHNST system.
\begin{problem}\label{prob:ac}
Let $(\phi_{h,0},\u_{h,0},\theta_{h,0})\in \Vh\times\Xh\times \Vh^+$ be given. Find the functions $(\phi_h,\u_h,\theta_h)\in \Pi^1_c(\Itau;\Vh \times\Xh\times \Vh^+)$ and $(\mu_h,\pi_h)\in \Pi^0(\Itau;\Vh\times\Vh)$  such that
\begin{align*}
  \la d_\tau^{n+1}\phi_h,\psi_h \ra &- \la \phi_h^*\u_h^{n+1/2},\nabla\psi_h\ra + \la \LL_{11}^*\nabla\mu_h^{n+1} - \LL_{12}^*\nabla\theta_h^{n+1},\nabla\psi_h \ra = 0, \\
  \la \mu_h^{n+1},\xi_h \ra &- \gamma\la \nabla\phi_h^{n+1},\nabla\xi_h \ra - \la \partial_\phi \Psi(\phi^{n+1}_h,\phi^{n}_h,\theta^{n+1}_h),\xi_h \ra = 0, \\
  \la d_\tau^{n+1}e_h,w_h \ra &  + \la \LL_{12}^*\nabla\mu_h^{n+1} - \LL_{22}^*\nabla\theta_h^{n+1}, \nabla w_h\ra - \la \eta^*\snorm{\Du_h^{n+1/2}}^2, w_h\ra - \eps \la\snorm{\div(\u_h^{n+1/2})}^2, w_h\ra\\
  & - \delta\la h^2\snorm{\nabla\pi_h^{n+1}}^2, w_h\ra- \la \bsig_h^*\cdot\u_h^{n+1/2},\nabla w_h\ra-\la \tfrac{\phi_h^*}{\theta_h^{n+1}}\nabla\mu_h^{n+1} - \bsig_h^*\tfrac{\nabla\theta_h^{n+1}}{\theta_h^{n+1}},\u_h^{n+1/2} w_h \ra \\
  &- \la (s_h^*+\phi_h^*\mu_h^*)\u_h^{n+1/2} ,\tfrac{\theta_h^{n+1}\nabla w_h-w_h\nabla\theta_h^{n+1}}{(\theta_h^*)^2} \ra=0, \\
   \la d_\tau^{n+1}\u_h,\vv_h \ra &+ \mathbf{c}_{skw}(\u_h^{*},\u_h^{n+1/2},\vv_h) + \la \eta^*\Du_h^{n+1/2},\mathrm{D}\vv_h \ra + \eps\la\div(\u_h^{n+1/2}),\div(\vv_h)\ra- \la \pi_h^{n+1},\div(\vv_h)\ra \\
  & +\la \tfrac{\phi_h^*}{\theta_h^{n+1}}\nabla\mu_h^{n+1} - \bsig_h^*\tfrac{\nabla\theta_h^{n+1}}{\theta_h^{n+1}} - (s_h^*+\phi_h^*\mu^*_h)\tfrac{\nabla\theta_h^{n+1}}{(\theta_h^*)^2},\vv_h \ra=0, \\
  \hspace{-8em}\la \div(\u_h^{n+1/2}),q_h \ra &=  -\delta\la h^2\nabla \pi^{n+1}_h,\nabla q_h \ra
\end{align*}
holds for $(\psi_h,\xi_h,w_h,\vv_h,q_h)\in\Vh\times\Vh\times\Vh^+\times \Xh\times\Vh$, with the stabilisation parameters $\eps$ and $\delta$ which may also depend on $h$, and $g^*$ denoting an evaluation of $g$ at any $t\in[t^n,t^{n+1}]$, but all terms have to be evaluated at the same point in time. 
\end{problem}

\begin{theorem}\label{thm:result}
For any solution $(\phi_h,\mu_h,\theta_h,\u_h,\pi_h)$ of Problem \ref{prob:ac} discrete mass and total energy conservation as well as entropy production holds, i.e.
\begin{align*}
&\la \phi^{n+1} -\phi^0,1 \ra = 0 ,\qquad \la \tfrac{1}{2}\snorm{\u^{n+1}_h}^2 + e(\phi_h^{n+1},\theta_h^{n+1}) -  \tfrac{1}{2}\snorm{\u^{0}_h}^2 - e(\phi_h^{0},\theta_h^{0},1 \ra =  0, \\
&\la  s(\phi_h^{n+1},\theta_h^{n+1}) -  s(\phi_h^{0},\theta_h^{0}),1\ra = \tau\sum_{k=0}^{n_T}\mathcal{D}_{\theta^{n+1}_h,\LL^*}(\u^{n+1}_h,\mu_n^{n+1},\theta_h^{n+1}) + \sum_{k=0}^{n_T}\mathcal{D}_{num}^{k+1},
\end{align*}
where the numerical dissipation satisfies $\mathcal{D}_{num}^{k+1}\geq0$ and is given by
\begin{align*}
    \mathcal{D}_{num}^{k+1} &= \tfrac{\gamma}{2}\norm{\nabla d^{k+1}\phi_h}^2 - \partial_{\theta\theta}\Psi(\phi^n_h,\xi^3_h)(d^{k+1}\theta_h)^2 + (\partial_{\phi\phi}\Psi_{vex}(\xi^1_h,\theta^{k+1}_h) - \partial_{\phi\phi}\Psi_{cav}(\xi^2_h,\theta^{k+1}_h))(d^{k+1}\phi_h)^2 \\
    & + \tau\eps\norm*{\sqrt{\theta_h^{k+1}}\div(\u_h^{k+1/2})}_0^2 +  \tau\delta\norm*{h\sqrt{\theta_h^{k+1}}\nabla \pi_h^{k+1}}_0^2
\end{align*}
\end{theorem}

\begin{proof}
Conservation of mass follows immediately by taking $\psi_h=1$. 

\textbf{Conservation of total energy:} 
Using the algebraic identity $a(a+b) - (a+b)b=a^2-b^2 $ we obtain after rearrangement
\begin{equation*}
    \frac{1}{\tau}\la \tfrac{1}{2}\snorm{\u^{n+1}_h} + e_h^{n+1} -  \tfrac{1}{2}\snorm{\u^{n}_h}^2 -e_h^n, 1\ra = \la d_\tau^{n+1}\u_h,\u_h^{n+1/2} \ra + \la d_\tau^{n+1}e_h,1 \ra.
\end{equation*}

We insert $\vv_h=\u_h^{n+1/2}, w_h=1$ using the skew-symmetric of $\mathbf{c}_{skew}(\u_h^*,\u_h^{n+1/2},\u_h^{n+1/2})=0$ and $\nabla 1=0$ to obtain
\begin{align*}
\hspace{-4em}\la d_\tau^{n+1}\u_h,\u_h^{n+1/2} \ra + \la d_\tau^{n+1}e_h,1 \ra &= - \la \eta^*\Du_h^{n+1/2},\Du_h^{n+1/2} \ra - \eps\la\div(\u_h^{n+1/2}),\div(\u_h^{n+1/2})\ra+ \la \pi_h^{n+1},\div(\u_h^{n+1/2})\ra \\
  & -\la \tfrac{\phi_h^*}{\theta_h^{n+1}}\nabla\mu_h^{n+1} - \bsig_h^*\tfrac{\nabla\theta_h^{n+1}}{\theta_h^{n+1}} - (s_h^*+\phi_h^*\mu^*_h)\tfrac{\nabla\theta_h^{n+1}}{(\theta_h^*)^2},\u_h^{n+1/2} \ra \\
& +\la \eta^*\Du_h^{n+1/2}, \Du_h^{n+1/2}\ra + \eps\la\div(\u_h^{n+1/2}), \div(\u_h^{n+1/2})\ra\\
  & + \delta\la h^2\nabla\pi_h^{n+1}, \nabla\pi_h^{n+1}\ra+\la \tfrac{\phi_h^*}{\theta_h^{n+1}}\nabla\mu_h^{n+1} - \bsig_h^*\tfrac{\nabla\theta_h^{n+1}}{\theta_h^{n+1}},\u_h^{n+1/2} \ra \\
  &- \la (s_h^*+\phi_h^*\mu_h^*)\u_h^{n+1/2} ,\tfrac{\nabla\theta_h^{n+1}}{(\theta_h^*)^2} \ra\\
  \intertext{Obvious cancellation yields}
  & = \la \pi_h^{n+1},\div(\u_h^{n+1/2})\ra + \delta\la  h^2\nabla\pi_h^{n+1}, \nabla\pi_h^{n+1}\ra
\end{align*}
The results follows then directly by taking $q_h=\pi_h^{n+1}$.

\textbf{Entropy production:} For the entropy production we compute
\begin{align*}
 \la s_h^{n+1}- s_h^n,1\ra =& \la \theta_h^{n+1}e_h^{n+1} - \theta_h^{n}e_h^n - \Psi_h^{n+1} + \Psi_h^n - \tfrac{\gamma}{2}\snorm{\nabla\phi_h^{n+1}}^2 + \tfrac{\gamma}{2}\snorm{\nabla\phi_h^{n}}^2,1\ra\\
 =&  \la d^{n+1}e_h,\theta_h^{n+1}\ra  - \gamma\la\nabla\phi_h^{n+1},\nabla d^{n+1}\phi_h\ra - \la\Psi_h^{n+1}-\Psi_h^n,1 \ra + \la e_h^n,d^{n+1}\theta_h\ra + \tfrac{\gamma}{2}\snorm{\nabla d^{n+1}\phi_h}^2. 
\end{align*}
Adding $\pm\tfrac{1}{\tau}\partial_\phi\Psi(\phi_h^{n+1},\phi_h^n,\theta_h^{n+1})d^{n+1}\phi_h$ and insertion of $\xi_h=d_\tau^{n+1}\phi_h$ yields
\begin{align*}
\la d_\tau^{n+1}s_h,1\ra=& \la d_\tau^{n+1} e_h,\theta_h^{n+1}\ra - \la\mu_h^{n+1},d_\tau^{n+1}\phi_h\ra - \tfrac{1}{\tau}\la \Psi_h^{n+1} - \Psi_h^n - e_h^n d^{n+1}\theta_h,1\ra \\
&+ \tfrac{1}{\tau}\la \partial_\phi\Psi(\phi_h^{n+1},\phi_h^n,\theta_h^{n+1}),d^{n+1}\phi_h\ra + \tfrac{\gamma}{2\tau}\norm{\nabla d^{n+1}\phi_h}_0^2  \\
 = & \la d_\tau^{n+1} e_h, \theta_h^{n+1}\ra - \la \mu_h^{n+1},d_\tau^{n+1}\phi_h \ra + \mathcal{D}_{num,a}^{n+1}.
\end{align*}

Next we consider the first two inner-products and insert $\psi_h=-\mu_h^{n+1}, w_h=\theta_h^{n+1}$ into the discrete formulation 
\begin{align*}
 \hspace{-3em}\la d_\tau^{n+1} e_h, \theta_h^{n+1}\ra - \la \mu_h^{n+1},d_\tau^{n+1}\phi_h \ra &=
 - \la \LL_{12}^*\nabla\mu_h^{n+1} - \LL_{22}^*\nabla\theta_h^{n+1}, \nabla \theta^{n+1}_h\ra + \la \eta^*\snorm{\Du_h^{n+1/2}}^2, \theta^{n+1}_h\ra \\
 &+ \eps\la\snorm{\div(\u_h^{n+1/2})}^2, \theta^{n+1}_h\ra + \delta\la  h^2\snorm{\nabla\pi_h^{n+1}}^2, \theta^{n+1}_h\ra- \la \bsig_h^*\cdot\u_h^{n+1/2},\nabla \theta^{n+1}_h\ra \\
 &+\la \tfrac{\phi_h^*}{\theta_h^{n+1}}\nabla\mu_h^{n+1} - \bsig_h^*\tfrac{\nabla\theta_h^{n+1}}{\theta_h^{n+1}},\u_h^{n+1/2} \theta^{n+1}_h \ra \\
  &+ \la (s_h^*+\phi_h^*\mu_h^*)\u_h^{n+1/2} ,\tfrac{\theta_h^{n+1}\nabla \theta^{n+1}_h-\theta^{n+1}_h\nabla\theta_h^{n+1}}{(\theta_h^*)^2} \ra \\
  &+ \la \phi_h^*\u_h^{n+1/2},\nabla\mu^{n+1}_h\ra - \la \LL_{11}^*\nabla\mu_h^{n+1} - \LL_{12}^*\nabla\theta_h^{n+1},\nabla\mu^{n+1}_h \ra \\
  \intertext{Cancellation yields}
  & = - \la \LL_{12}^*\nabla\mu_h^{n+1} - \LL_{22}^*\nabla\theta_h^{n+1}, \nabla \theta^{n+1}_h\ra + \la \eta^*\snorm{\Du_h^{n+1/2}}^2, \theta^{n+1}_h\ra \\
 &+ \eps\la \snorm{\div(\u_h^{n+1/2})}^2, \theta^{n+1}_h\ra + \delta\la h^2\snorm{\nabla\pi_h^{n+1}}^2, \theta^{n+1}_h\ra \\
 &- \la \LL_{11}^*\nabla\mu_h^{n+1} - \LL_{12}^*\nabla\theta_h^{n+1},\nabla\mu^{n+1}_h \ra \\
 &= \la \LL^*(\nabla\mu_h^{n+1},\nabla\theta_h^{n+1})^\top,(\nabla\mu_h^{n+1},\nabla\theta_h^{n+1})^\top \ra + \la \eta^*\theta_h^{n+1},\snorm{\Du_h^{n+1/2}}^2 \ra \\
 &+ \eps\la\theta_h^{n+1},\snorm{\div(u_h^{n+1/2})}^2 \ra + \delta\la h^2\theta_h^{n+1},\snorm{\nabla\pi_h^{n+1}}^2 \ra \\
 & = \mathcal{D}_{\LL^*,\theta^{n+1}_h}(\u_h^{n+1/2},\mu_h^{n+1},\theta_h^{n+1}) + \mathcal{D}_{num,b}^{n+1}.
\end{align*}

Finally, for the numerical dissipation we add  $\pm \tfrac{1}{\tau}\Psi_h(\phi_h^n,\theta_h^{n+1})$ which yields
\begin{align*}
 \mathcal{D}_{num,a}^{n+1} &= \tfrac{1}{\tau}\int_\Omega -\Psi_h^{n+1} +\Psi_h(\phi_h^n,\theta_h^{n+1}) -  e_h^nd^{n+1} \theta_h + \tfrac{\gamma}{2}\snorm{\nabla d^{n+1}\phi_h}^2\\
& - \Psi_h(\phi_h^n,\theta_h^{n+1}) - \Psi_h^n  + \partial_\phi\Psi(\phi_h^{n+1},\phi_h^n,\theta_h^{n+1})d^{n+1}\phi_h \\
& = \tfrac{1}{\tau}\int_\Omega \tfrac{\gamma}{2}\norm{\nabla d^{n+1}\phi_h^{n+1}}^2 - \partial_{\theta\theta}\Psi(\phi^n_h,\xi^3_h)(d^{n+1}\theta_h)^2 \\
    &+ (\partial_{\phi\phi}\Psi_{vex}(\xi^1_h,\theta^{n+1}_h) - \partial_{\phi\phi}\Psi_{cav}(\xi^2_h,\theta^{n+1}_h))(d^{n+1}\phi_h^{n+1})^2,
\end{align*}
where $\xi^1_h,\xi^2_h$ are convex combinations of $\phi_h^{n+1},\phi_h^n$ and $\xi^3_h$ is a convex combination of $\theta_h^{n+1},\theta_h^n$. Using the structural assumptions on the potential $\Psi$, cf. (A4), we see that $\mathcal{D}_{num,a}^{n+1}\geq 0$ follows directly. The result then follows by setting $\mathcal{D}_{num}^{n+1}=\mathcal{D}_{num,a}^{n+1} + \mathcal{D}_{num,b}^{n+1}$ and summation over $k$.
\end{proof}

If conservation of total energy is relaxed by total energy dissipation, one can replace all $\u_h^{n+1/2}$ by $\u_h^{n+1}$ and the corrections for the stabilisation in the internal energy can be neglected.

\section{Numerical Test}
In all tests we consider the case of $g^*=g^{n}$ and the resulting nonlinear systems are solved with a Newton method with tolerance $10^{-12}$ in two dimensions and $10^{-8}$ in three dimensions. 
\subsection{Convergence test}

For the convergence test we set $\Omega=(0,1)^2$ and $T=0.1$ which is identified with the two-torus $\mathbb{T}^2$. This accounts for the periodic boundary conditions. 
We consider the initial data
\begin{align*}
 \phi_0(x,y) &= 0.4 + 0.2\sin(2\pi x)\sin(2\pi y), \qquad \theta_0(x,y)= 1 + 0.2\sin(2\pi x)\sin(2\pi y)   \\
 \u_0(x,y) &= 10^{-2}(-\sin(\pi x)^2\sin(2\pi y),\sin(2\pi x)\sin(\pi y)^2)
\end{align*}
with the set of functions and parameters
\begin{align*}
\tilde\Psi(\phi,\theta) &= \log(\theta) + (2\theta-1)\phi^2(1-\phi)^2 + \frac{\gamma}{2}\snorm{\nabla\phi}^2,\\
e &=\tfrac{1}{\theta} + 2\phi^2(1-\phi)^2, \qquad  s = 1-\log(\theta) + \phi^2(1-\phi)^2 - \tfrac{\gamma}{2}\snorm{\nabla\phi}^2, \\
\gamma&=10^{-3},\qquad \eta=10^{-3} + \tfrac{1}{40}(\phi+1)^2,\qquad \LL=10^{-2}\cdot \mathbf{I},\qquad \eps=10^1,\qquad\delta=1.
\end{align*}
We consider the error in space at the final time $T=0.1$ with an respective step-size of $\tau = 10^{-3}$ and compare the  numerical solutions $(\phi_{h,\tau},\mu_{\rho,h,\tau},\theta_{h,\tau},\u_{h,\tau},p_{h,\tau})$ with those computed on uniformly refined grids, $(\phi_{h/2,\tau},\mu_{\rho,h/2,\tau},\theta_{h/2,\tau},\u_{h/2,\tau},p_{h/2,\tau})$, since no analytical solution is available.
The error quantities for the fully-discrete scheme are given in the energy-norm, i.e.
\begin{align*}
	\hspace{-1em}e^a_{h,\tau} &= \norm*{\phi_{h,\tau} - \phi_{h/2,\tau}}_{H^1}^2  + \norm*{\u_{h,\tau} - \u_{h/2,\tau}}_{L^2}^2  + \norm*{\theta_{h,\tau} - \theta_{h/2,\tau}}_{L^2}^2\\
    \hspace{-1em}e^b_{h,\tau} &=  \norm*{\mu_{h,\tau} - \mu_{h/2,\tau}}_{H^1}^2 + \norm*{\u_{h,\tau} - \u_{h/2,\tau}}_{H^1}^2 + \norm*{\theta_{h,\tau} - \theta_{h/2,\tau}}_{H^1}^2 
\end{align*}
as well as some separated errors, $e^\mu_{h,\tau}$, $e^{\u}_{h,\tau}$ and $e^{\theta}_{h,\tau}$, which denote the related single quantities from the second error norm. For the step sizes  $h_k=2^{-k}$ for $k=2,\ldots,6$ we get the following results.

\vspace{-1em}
\begin{table}[htbp!]
	\centering
	\small
	\caption{ Errors and experimental orders of convergence for the CHNST system with $g^*=g^n$.} 
	\begin{tabular}{c||c|c|c|c|c|c|c|c|c|c}
		$ k $ & $e^a_{h,\tau}$ & eoc & $e^b_{h,\tau}$ & eoc & $e^\mu_{h,\tau}$ & eoc & $e^\u_{h,\tau}$ & eoc & $e^\theta_{h,\tau}$ & eoc  \\
		\hline
		$ 2 $ & $3.02 \cdot 10^{-1}$ & --- & $3.87 \cdot 10^{-1}$ & --- & $1.97 \cdot 10^{-1}$ & --- & $7.37 \cdot 10^{-4}$  &   ---    & $1.89 \cdot 10^{-1}$  &   ---  \\
		$ 3 $ & $9.76 \cdot 10^{-2}$ & 1.63 & $1.32 \cdot 10^{-1}$ & 1.55 & $6.67 \cdot 10^{-2}$ & 1.56 & $1.67 \cdot 10^{-4}$  &   2.14   & $6.50 \cdot 10^{-2}$  & 1.54 \\
		$ 4 $ & $2.27 \cdot 10^{-2}$ & 2.11 & $3.35 \cdot 10^{-2}$  & 1.98 & $1.18 \cdot 10^{-2}$ & 1.91 & $5.40 \cdot 10^{-5}$  &   1.63   & $1.57 \cdot 10^{-2}$  & 2.05 \\
		$ 5 $ & $5.45 \cdot 10^{-3}$ & 2.06 & $8.26 \cdot 10^{-3}$  & 2.02 & $4.36 \cdot 10^{-3}$ & 2.02 & $1.87 \cdot 10^{-5}$  &   1.53   & $3.88 \cdot 10^{-3}$  & 2.02 \\
        $ 6 $ & $1.34 \cdot 10^{-3}$ & 2.03 & $2.05 \cdot 10^{-3}$  & 2.01 & $1.08 \cdot 10^{-3}$ & 2.02 & $3.26 \cdot 10^{-6}$  &   2.53   & $9.66 \cdot 10^{-4}$  & 2.00 \\
	\end{tabular}
\end{table}

In the two-dimensional test the mass and total energy conservation error was always in the range of at least $10^{-10}$, i.e. below the Newton tolerance, and entropy was increasing over time as predicted by Theorem \ref{thm:result}. Furthermore, choosing $g^*=g^{n+1}$ results almost the same error rates.

\subsection{Three dimensional simulations}

Due to the usage of $P_1|P_1$ elements computations in three space dimension are much more feasible then with the usual Taylor-Hood elements. As an example we consider an problem-adapted Taylor-Green vortex given by
\begin{align*}
 \phi_0(x,y,z) &= 0.4 + 0.2\sin(\pi x)\sin(2\pi y)\sin(2\pi z), \qquad \theta_0(x,y,z)= 1 + 0.6\sin(2\pi x)\sin(\pi y) \sin(2\pi z),  \\
 \u_0(x,y,z) &= (A\cos(2\pi x)\sin(2\pi y)\sin(2\pi z),B\sin(2\pi x)\cos(2\pi y)\sin(2\pi z),C\sin(2\pi x)\sin(2\pi y)\cos(2\pi z))
\end{align*}
with $A=B=2,~C=-4$ as initial data and changed parameters $\eta=10^{-4}$ and $\LL_{22}=10^{-4}$.
\begin{figure}[htbp!]
\centering
\footnotesize
\begin{tabular}{cc}
     \includegraphics[trim={38cm 0cm 37cm 9cm},clip,scale=0.115]{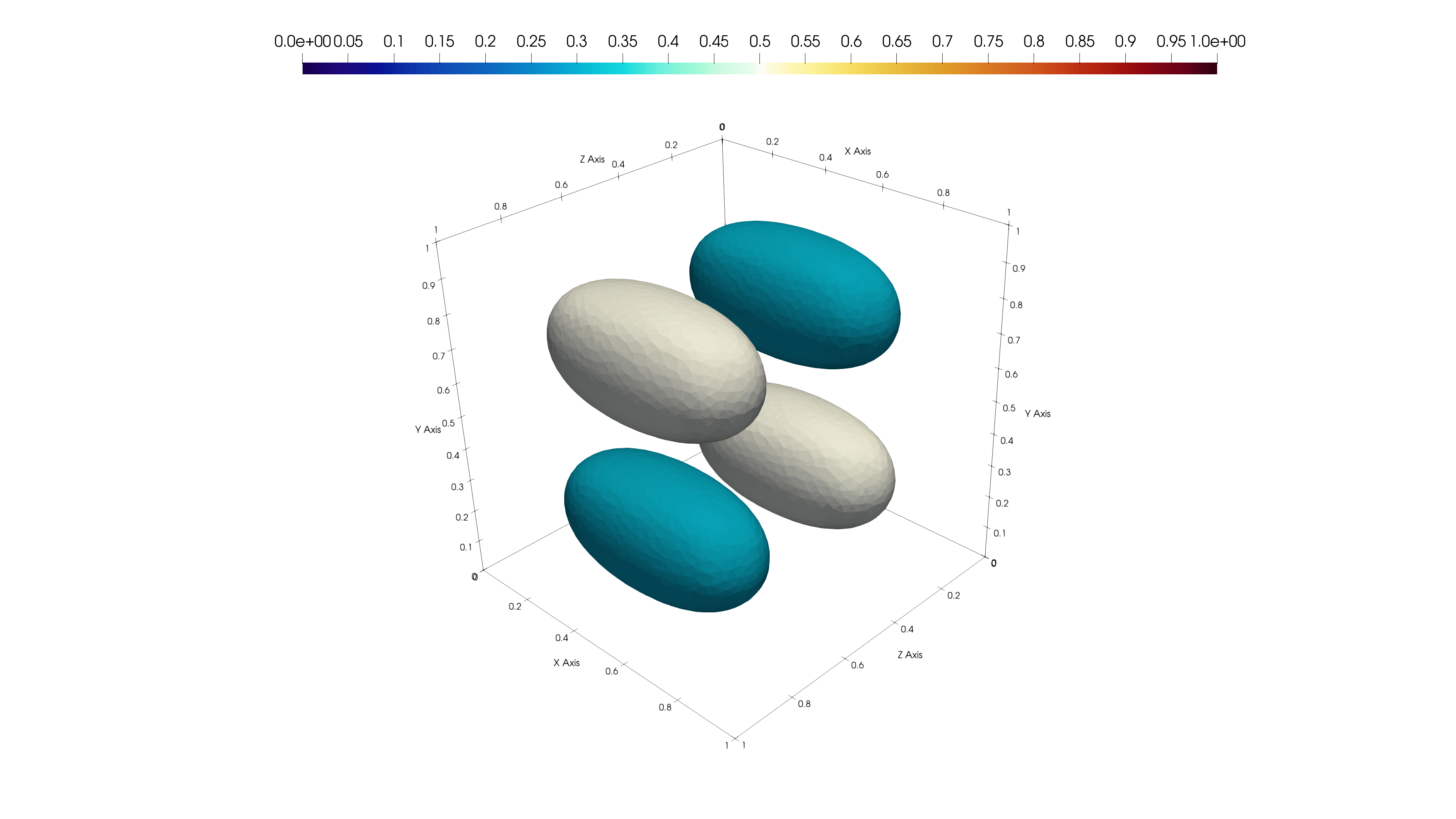} 
    &
    \includegraphics[trim={38cm 0cm 37cm 9cm},clip,scale=0.115]{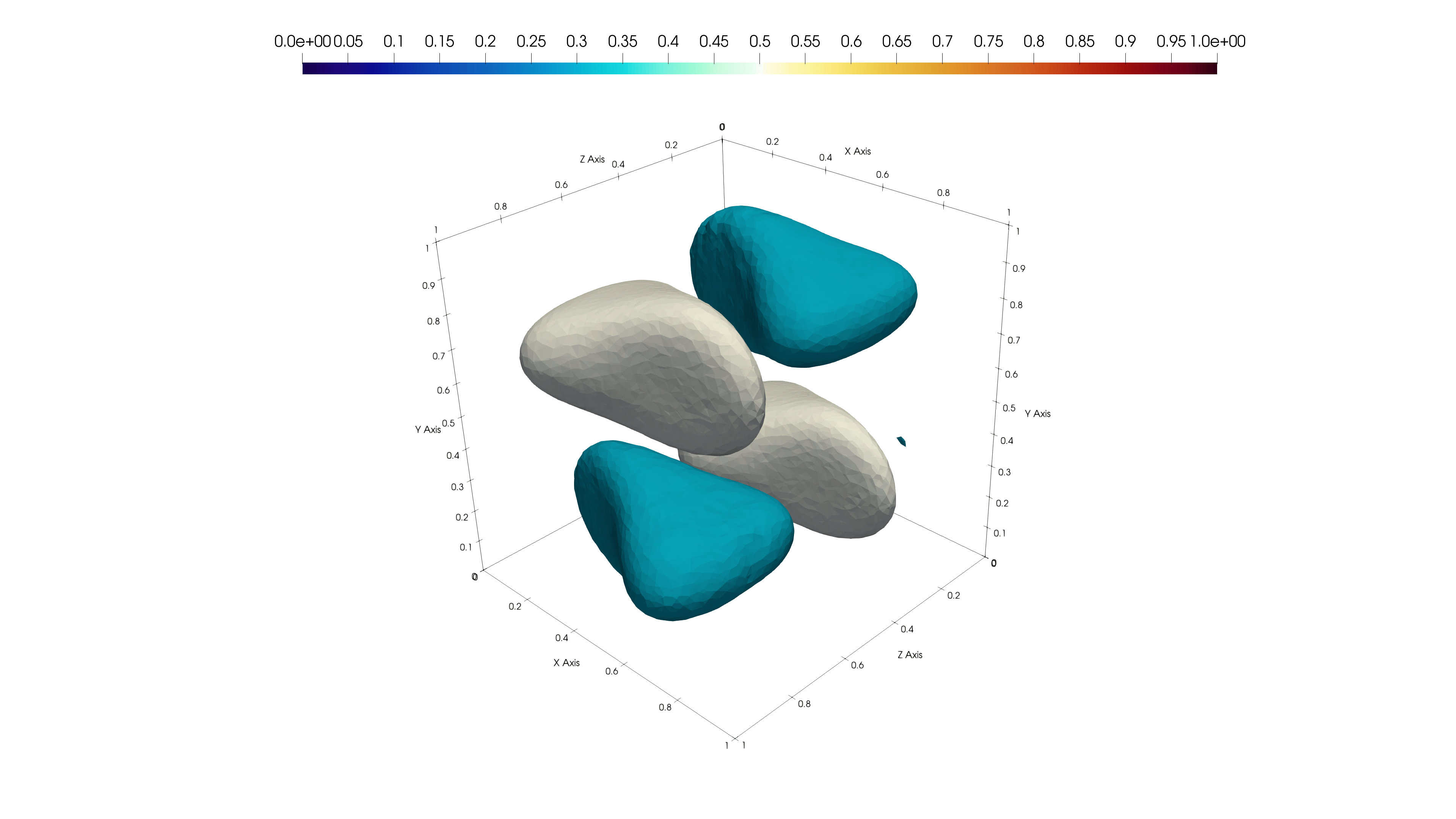} \\[-1em]
    
     \includegraphics[trim={38cm 0cm 37cm 9cm},clip,scale=0.115]{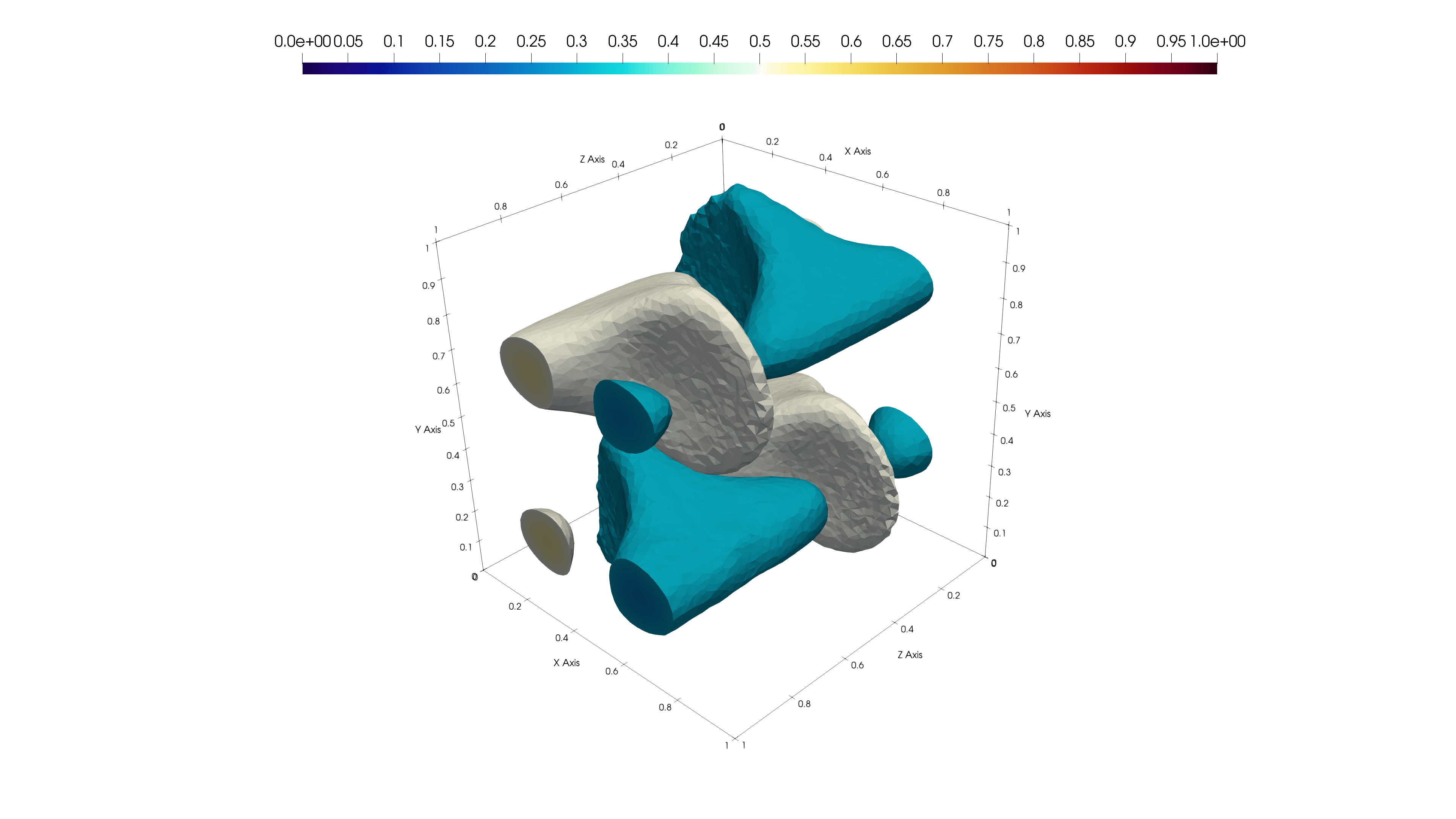} 
    &
    \includegraphics[trim={38cm 0cm 37cm 9cm},clip,scale=0.115]{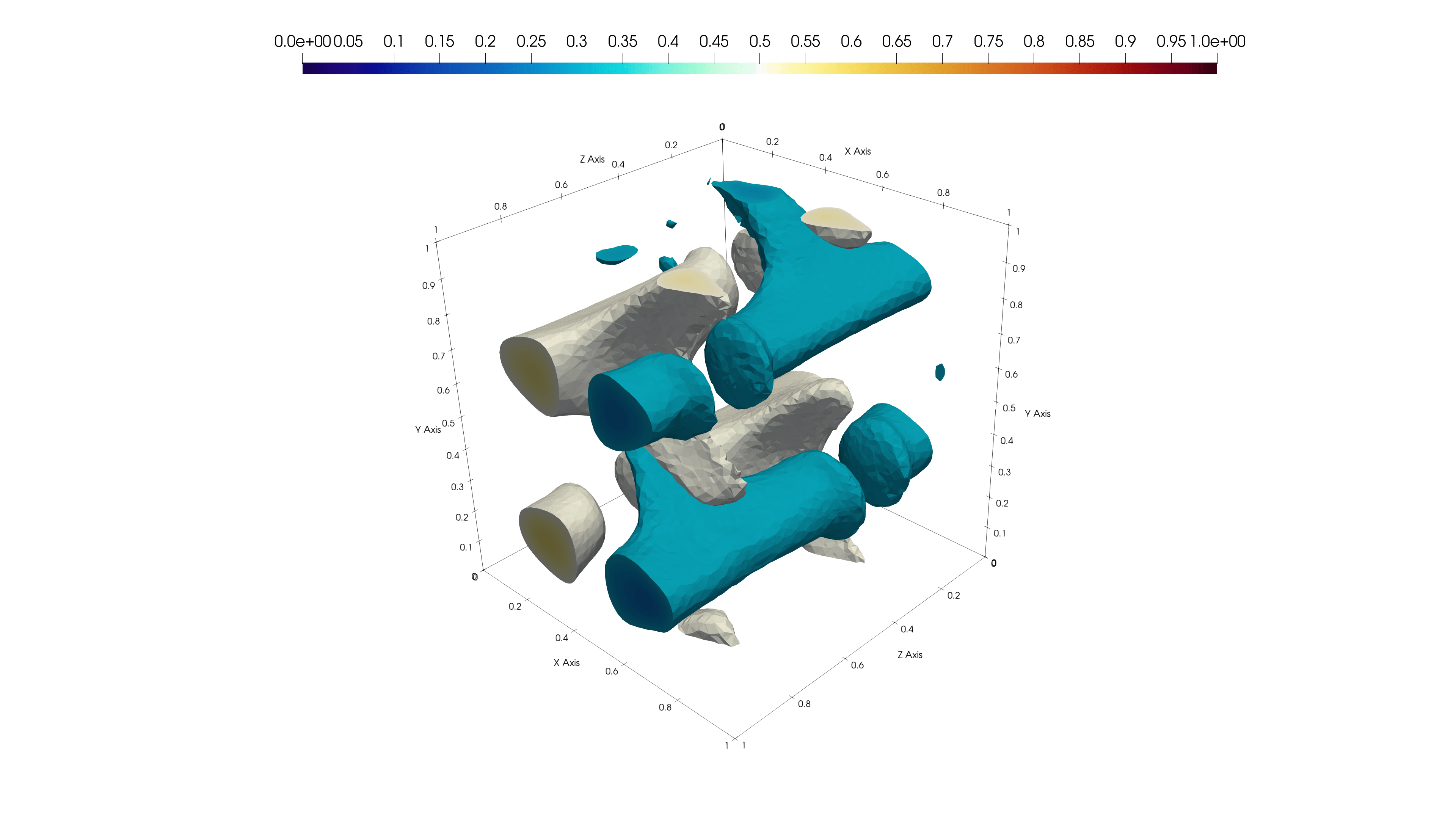} \\
    \multicolumn{2}{c}{\includegraphics[trim={25cm 66cm 16cm 0cm},clip,scale=0.17]{pics/phi_only_new.0000.png}}
\end{tabular}
\caption{Snapshots of the phase variable $\phi$ at the times $t\in\{0,0.03,0.06,0.1\}$.}
\label{pic:phi_only}
\end{figure}

\begin{figure}[htbp!]
\centering
\footnotesize
\begin{tabular}{cc}
     \multicolumn{2}{c}{\includegraphics[trim={14cm 50cm 10cm 1cm},clip,scale=0.335]{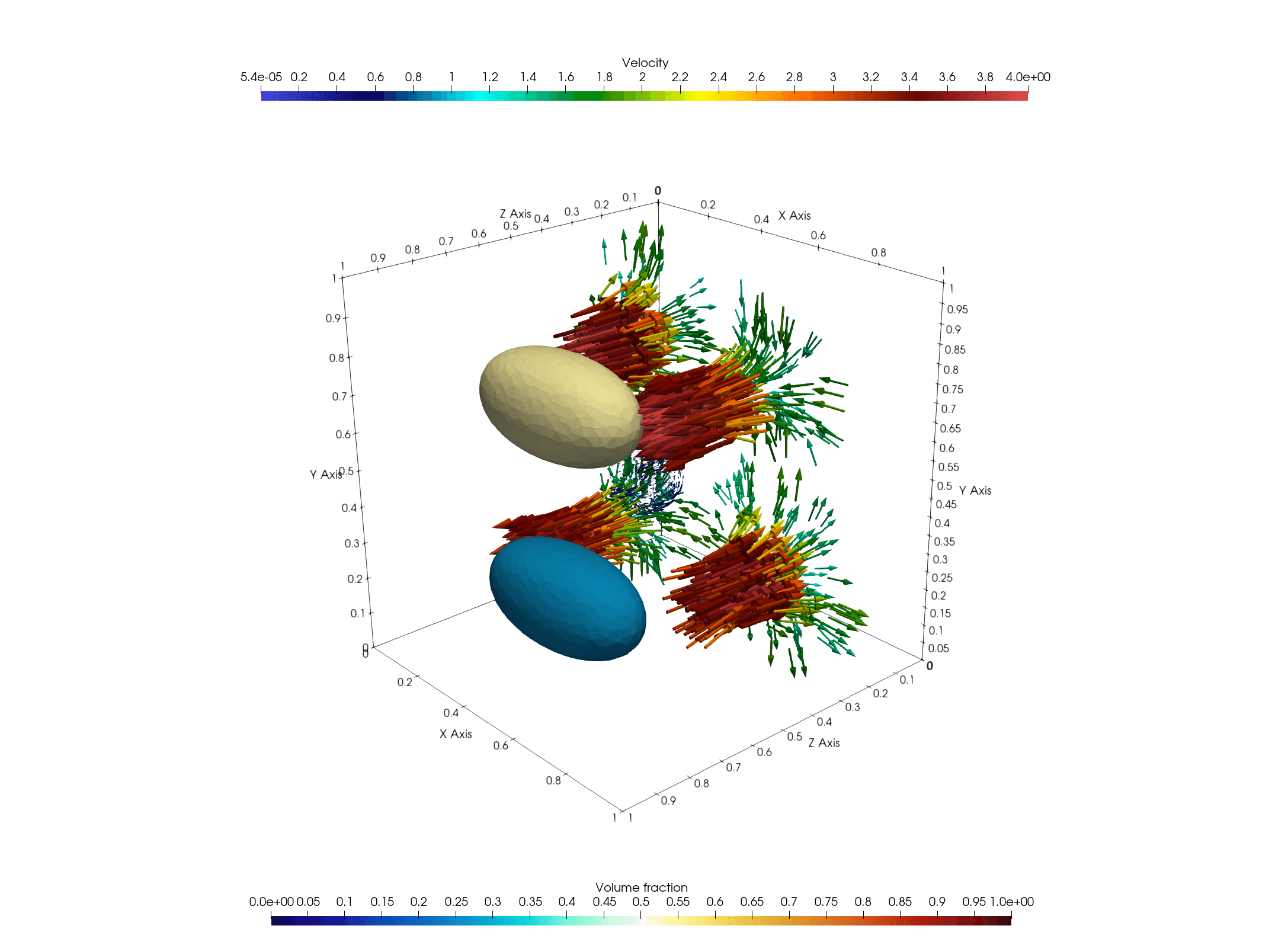}} \\[-0.5em]
     \includegraphics[trim={18cm 5cm 17.5cm 9cm},clip,scale=0.19]{pics/phi_andvelo.0000.png} 
    &
    \includegraphics[trim={18cm 5cm 17.5cm 9cm},clip,scale=0.19]{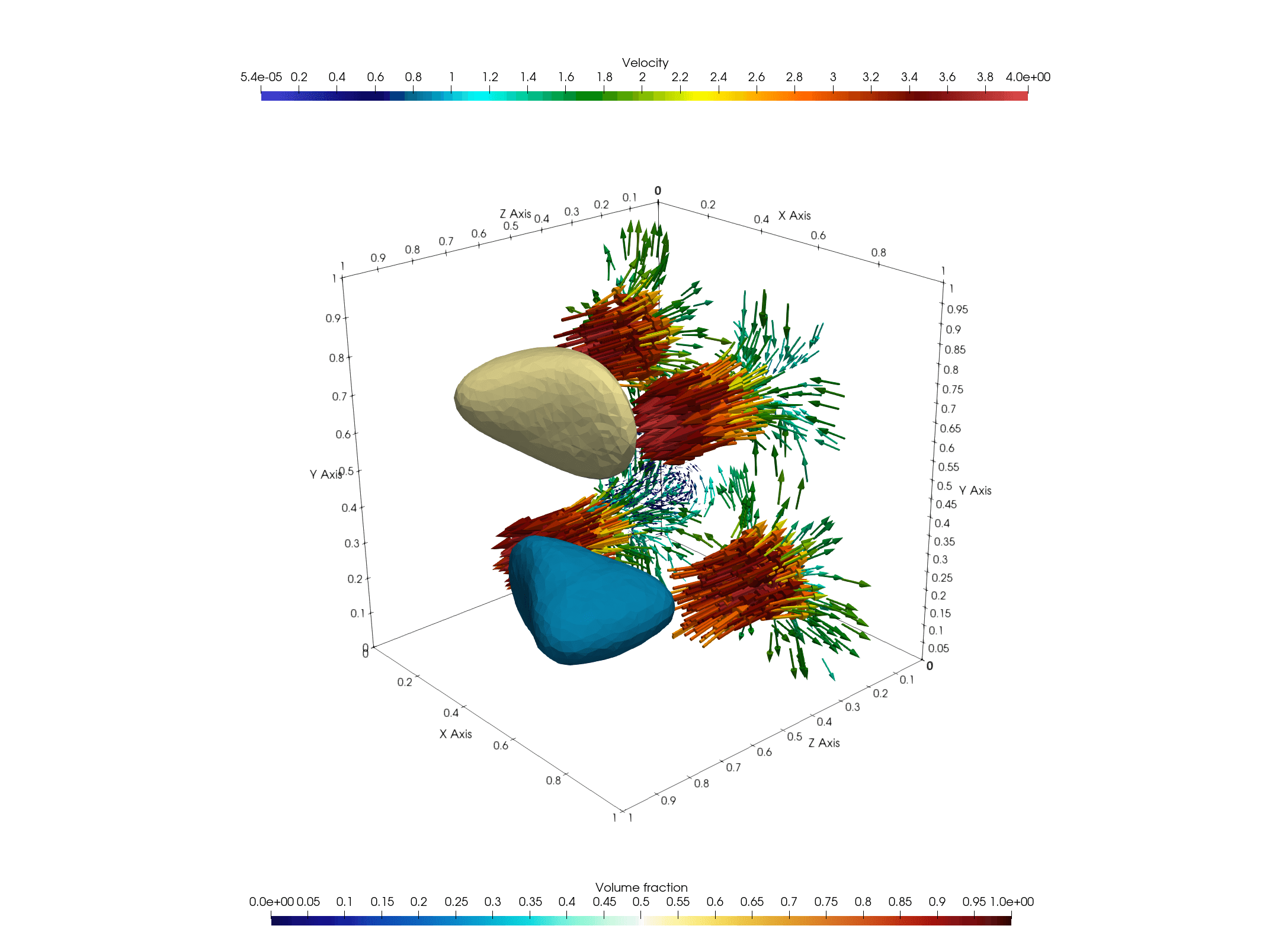} \\[-1em]
    
     \includegraphics[trim={18cm 5cm 17.5cm 9cm},clip,scale=0.19]{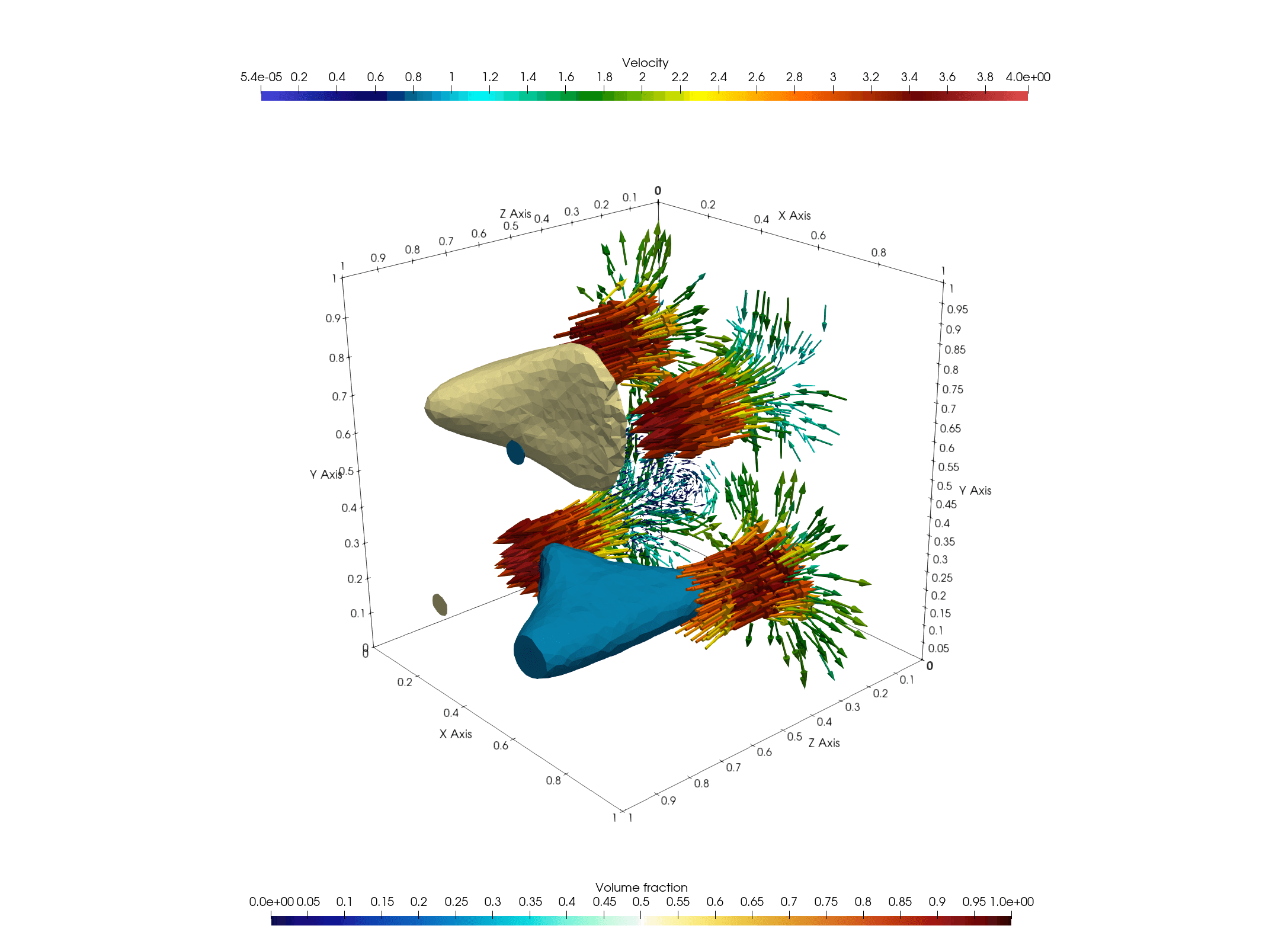} 
    &
    \includegraphics[trim={18cm 5cm 17.5cm 9cm},clip,scale=0.19]{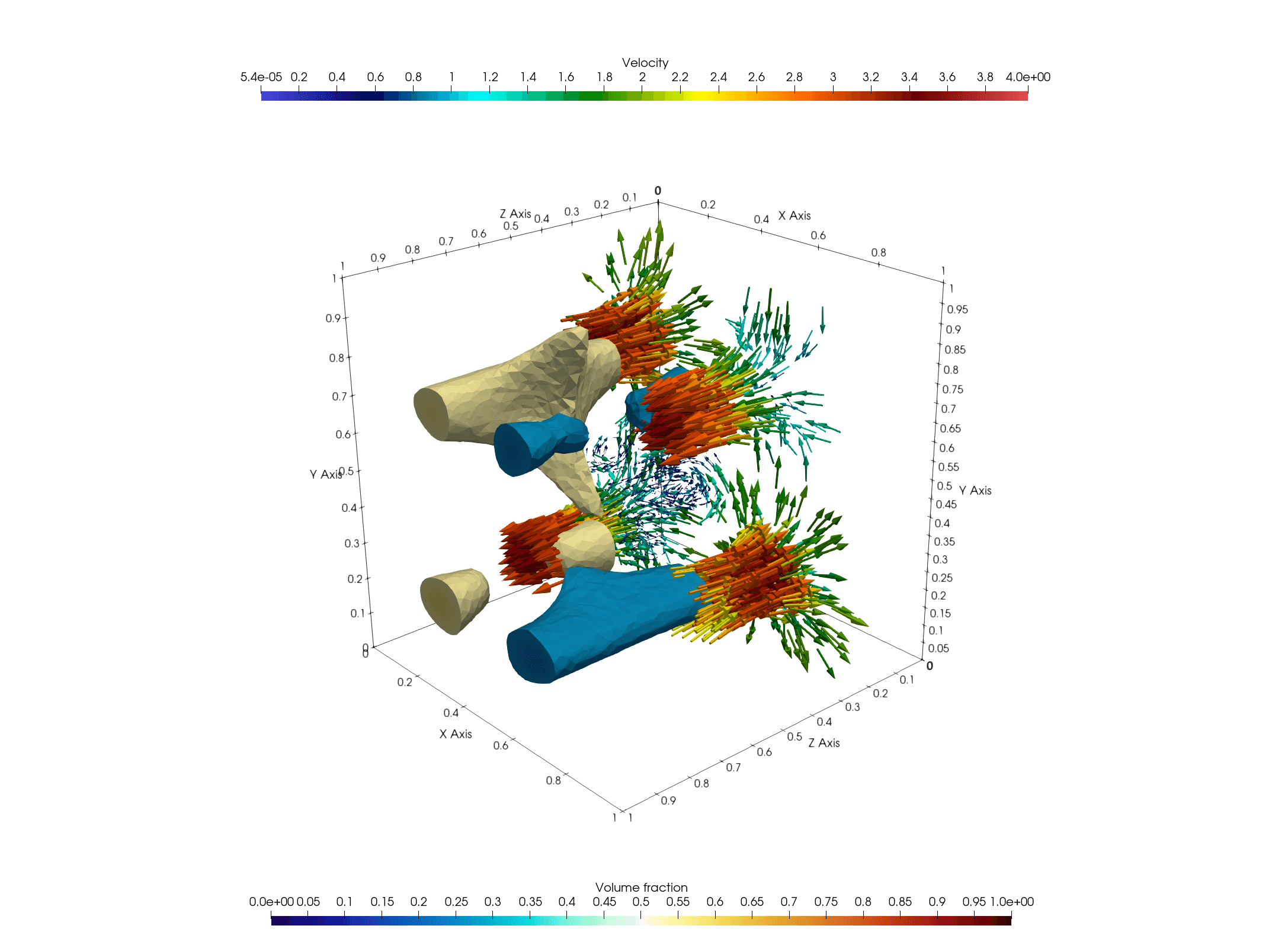} \\[-2em]
    \multicolumn{2}{c}{\includegraphics[trim={14cm 1cm 10cm 50cm},clip,scale=0.335]{pics/phi_andvelo.0000.png}} \\[-0.5em]
\end{tabular}
\caption{Snapshots of the phase variable $\phi$ and the velocity field $\u$ at the times $t\in\{0,0.03,0.06,0.1\}$.  }
\label{pic:phi_andvelo}
\end{figure}
The outcomes of simulations conducted with a maximum mesh size of $h=2.5\cdot10^{-2}$ and time step $\tau=10^{-3}$ are illustrated through various temporal snapshots in Figures \ref{pic:phi_only} to \ref{pic:phi_andvelo}. Throughout the evolution, the phase-field exhibits minimal further separation due to the pronounced mixing effects induced by the velocity field. Notably, a transition occurs where rod-like structures emerge gradually, presumably initiating conventional phase separation after a decay in the velocity field. The gradual decay of the velocity field, owing to its low viscosity, contributes to this process. In the context of three-dimensional simulations, mass conservation is maintained within an order of $10^{-14}$, while the error in total energy conservation is approximately $10^{-8}$, akin to the Newton tolerance level. As anticipated by Theorem \ref{thm:result}, entropy exhibits a continual increase over time.

\section{Conclusion \& Outlook}

This study introduces a fully discrete finite element approach for solving the non-isothermal Cahn-Hilliard-Navier-Stokes system described by equations \eqref{eq:ac1} through \eqref{eq:ac3}. The devised scheme is characterized by its structure-preserving nature, ensuring the conservation of mass, total energy, and entropy production. Notably, it facilitates the utilization of unstable finite element pairs for the incompressible Navier-Stokes equation through Brezzi-Pitkäranta stabilization, supplemented by the conventional Grad-Div stabilization. Experimental validation demonstrates the scheme's efficacy, enabling more efficient simulations compared to the original methodology outlined in \cite{brunk2024structurepreserving}.

Future endeavors will explore discretization strategies based on the original temperature instead of the inverse temperature. Furthermore, additional stabilization such as the Streamline-Upwind/Petrov-Galerkin (SUPG) method proposed by John et al. \cite{John2016} will be investigated to further enhance the scheme's performance.

\begin{acknowledgement}
Support by the German Science Foundation (DFG) via SPP~2256: 
 \textit{Variational Methods for Predicting Complex Phenomena in Engineering Structures and Materials} (project BR~7093/1-2) and via TRR~146: \textit{Multiscale Simulation Methods for Soft Matter Systems} (project C3) is gratefully acknowledged.
\end{acknowledgement}

\bibliographystyle{unsrt}
\bibliography{lit.bib}

\end{document}